\documentclass[12pt, twoside,a4paper, bibliography=totoc]{article}

\usepackage[top=3.5cm, bottom=3.5cm,includeheadfoot]{geometry}

\usepackage{graphicx} 
\usepackage{paralist} 
\usepackage{framed} 
\usepackage[T1]{fontenc} 
\usepackage[utf8]{inputenc} 
\usepackage[german=quotes]{csquotes} 
\usepackage{amsmath}
\usepackage{ntheorem} 
\usepackage{amssymb}
\usepackage{stmaryrd}
\usepackage{holtpolt}
\usepackage{lastpage} 
\usepackage{fancyhdr} 
\usepackage{lmodern} 
\usepackage[format=plain, justification=centering, font = {it, small}]{caption}
\usepackage{setspace}
\usepackage[arrow, matrix, curve]{xy}

\newtheorem{satz}{Satz}[section]

\newtheorem{lemma}[satz]{Lemma}
\newtheorem{definition}[satz]{Definition}
\newtheorem{corollary}[satz]{Corollary}
\newtheorem{theorem}[satz]{Theorem}
\newtheorem{remark}[satz]{Remark}
\newtheorem{example}[satz]{Example}

\newcommand{\qed}{\hfill $\blacksquare$}
\newenvironment{proof}{ \textit{Proof.}}{\qed\vspace{0.35cm}}

\newcommand{\R}{\mathbb R}
\newcommand{\K}{\mathcal K}
\newcommand{\N}{\mathbb N}

\newcommand{\Q}{\mathbb Q}
\newcommand{\Z}{\mathbb Z}

\newcommand{\B}{\mathcal B}

\newcommand{\thema}{Aspects of $p$-adic operator algebras}

\title{Aspects of $p$-adic operator algebras}
\author{Anton Clau\ss nitzer and Andreas Thom}

\date{\today}

\begin{document}

\maketitle

\begin{abstract}
In this article, we propose a $p$-adic analogue of complex Hilbert space and consider generalizations of some well-known theorems from functional analysis and the basic study of operators on Hilbert spaces. We compute the $K$-theory of the analogue of the algebra of compact operators and the algebra of all bounded operators. This article contains a survey on results from the thesis of the first author.
\end{abstract}


\section{Introduction}\label{introduction}

 While there exists a rich literature on $p$-adic functional analysis in general (cf.  Schneider's book \cite{schneider} as a comprehensive source), it seems that only few publications treat $p$-adic operator algebras, their $K$-theory and their application to group rings. In the following article, the authors want to give their contribution to the subject with focus on an $p$-adic analogue of the classical Hilbert space featuring phenomena such as self-duality etc.
This $p$-adic \emph{Hilbert space} $\Q_p(X)$ (sometimes called the restricted product of $\Q_p$ indexed by $X$) is defined as the set of all maps $\xi\colon X\rightarrow\Q_p$ such that $|\xi(x)|_p>1$ holds for only finitely many elements $x\in X$. The space $\Q_p(X)$ is not a $\Q_p$-vector space, but, equipped with the canonical addition, scalar multiplication with scalars in $\Z_p$ and an appropriate topology $\tau$, a locally compact topological $\Z_p$-module. We will introduce a \emph{scalar product} $\langle\ \,, \ \rangle\colon\Q_p(X)\times\Q_p(X)\rightarrow S^1$ on $\Q_p(X)$. It turns out that the Pontryagin dual of $\Q_p(X)$ is isomorphic to $\Q_p(X)$ as a topological group and all characters can be uniquely represented by scalar product with an element of $\Q_p(X)$, this correspondence yielding the isomorphism of $\Q_p(X)$ with its dual. As in the usual Archimedean case, one can define the algebra $\B(\Q_p(X))$ of continuous $\Z_p$-linear operators on $\Q_p(X)$. Using the notion of adjoint operators (cf. Section \ref{adjoint}) and of the operator norm (cf. Section \ref{norm topology}), $\B(\Q_p(X))$ can be given the structure of a complete normed $*$-algebra over $\Z_p$, i.\,e. a Banach-$*$-algebra over $\Z_p$. In analogy with the Archimedean case, it is possible to define a continuous functional calculus for certain operators in $\B(\Q_p(X))$, the so-called \emph{normal contractions} (cf. Section \ref{functional calculus}). The definition is based on Mahler's representation theorem of the continuous functions $\Z_p\rightarrow\Z_p$ as infinite $\Z_p$-linear combinations of binomial coefficients. 

It is possible to define an analogue $\K(\Q_p(X))$ for the ideal of compact operators in a Hilbert space (cf. Section \ref{compact operators in our context}). Furthermore, we will introduce and study a matrix representation of operators in $\B(\Q_p(X))$ (Section \ref{matrix representation}).

In Section \ref{some results on idempotents}, we will be interested in idempotents of the ring $\B(\Q_p(X))$ and its $K$-theory, namely its $K_0$-group. Compared to the usual case of projections in the complex Hilbert space, idempotents and projections in $\B(\Q_p(X))$ are much harder to study). For example, there is in general no projection onto the intersection of images of two given projections etc. But at least, we will show that, as in the Archimedean case, we have the isomorphisms $K_0(\K(\Q_p(X)))\cong\Z$ and $K_0(\B(\Q_p(X)))=0$ (cf. Section \ref{compact K-group} and Section \ref{vanishing K-group}). Interestingly, also the fact that each idempotent in the quotient algebra $\B(\Q_p(X))/\K(\Q_p(X))$ can be lifted to an idempotent in $\B(\Q_p(X))$ remains true, but the proof is different from the analogous Archimedean theorem (cf. Section \ref{Calkin lifting}).

 This article is a short version of the first three chapters in the thesis of one of the authors (cf. \cite{claussnitzer}), and most parts are taken from there. In the last two chapters of \cite{claussnitzer}, the reader can find additional considerations, e.\,g. on the definition of the tensor product of operator algebras acting on $\Q_p(X)$, the application of our approach to the case that $X=\Gamma$ is a countable group, the $p$-adic analogue of the group von Neumann algebra etc.

\section{The $p$-adic analogue of a Hilbert space}
\subsection{The space $\Q_p(X)$ and the topology $\tau$}\label{our space}

 Let $X$ be a countable set. Consider the set 
\begin{align*}
\mathbb{Q}_p(X):=\{\xi\colon X\rightarrow\Q_p; |\xi(i)|_p\leq1\text{ for all but finitely many }i\in X\}. 
\end{align*}
On this set, we define a topology $\tau$ by saying that a set $A\subseteq\Q_p(X)$ is open if for all $P\subseteq X$ with $|P|<\infty$, the set 
\begin{align*}
\left(\prod_{i\in P}\Q_p\times\prod_{j\in X\setminus P}\Z_p\right)\cap A 
\end{align*}
is open in $\prod_{i\in P}\Q_p\times\prod_{j\in X\setminus P}\Z_p$ with respect to the product topology. Note that $\tau$ is the largest topology on $\Q_p(X)$ such that all the inclusions of the form 
\begin{align*}
\prod_{i\in P}\Q_p\times\prod_{j\in X\setminus P}\Z_p\hookrightarrow \Q_p(X) 
\end{align*}
with finite $P\subseteq X$ are continuous. 

Using the terminology of \cite{neukirchschmidtwingberg}, Def. I.1.1.12, the space $\Q_p(X)$ is called the \emph{restricted product} of countably many copies of $\Q_p$ with respect to the open subgroups $\Z_p\subseteq\Q_p$. 

Furthermore, notice the similarity of this construction with the construction of the \emph{adele-rings} in \cite{maninpanchishkin}, chapter 4.3.7. 

 For $x\in X$, we define the element $\delta_x\in\Q_p(X)$ by $\delta_x(x)=1$ and $\delta_x(y)=0$ for $y\in X\setminus\{x\}$.
 
 The following lemma is easy to prove: 

\begin{lemma}\label{convergence in our Hilbert space}
With respect to $\tau$, a sequence $(\xi_n)_{n\in\N}$ in $\Q_p(X)$ converges to $\xi\in\Q_p(X)$ if and only if it converges entrywise to $\xi$ and if the set $\{x\in X; \exists n\in\N\colon|\xi_n(x)|_p>1\}$ is finite. 
\end{lemma}

 Equipped with the natural coordinate-wise addition and the topology $\tau$, the set $\Q_p(X)$ becomes a locally compact and $\sigma$-compact Hausdorff topological abelian group where the subset 
\begin{align*}
\Z_p(X):=\prod_{i\in X}\Z_p 
\end{align*}
is (according to Tychonoff's theorem) a compact open subgroup. The group $\Q_p(X)$ additionally carries a natural structure of a $\Z_p$-module, but it is not a $\Q_p$-vector space if $X$ is infinite.

The topological groups $\Q_p(X)$ have already been considered in \cite{rajagopalansoundararajan} where the authors show that all self-dual (in Pontryagin's sense) metrizable locally compact torsion-free abelian groups are either of this form or of the form $\R^n$, of the form $D\oplus \widehat D$ where $D$ is a countable torsion-free divisible discrete group, or a (local) direct sum of groups of these types.
 
Also the following lemma is easy to verify: 

\begin{lemma}\label{Polish group}
The abelian group $\Q_p(X)$ is a Polish group. 
\end{lemma}

\begin{definition}\label{bounded linear operators}
The set of all $\Z_p$-linear $\tau$-continuous operators on $\Q_p(X)$ is denoted by $\B(\Q_p(X))$. 
\end{definition}

Note that a $\tau$-continuous group homomorphism $A\colon\Q_p(X)\rightarrow\Q_p(X)$ is already in $\B(\Q_p(X))$.  The set $\B(\Q_p(X))$ forms a $\Z_p$-module with the canonical operations.

The following two lemmas are special cases of well-known versions of the open mapping and closed graph theorems for certain topological groups (cf. \cite{hofmannmorris}, Theorem 1.5 and \cite{kelley}, p. 213). For these useful lemmas, the assumption on $X$ to be countable becomes relevant. 

\begin{lemma}\label{closed graph theorem}
Let $A\colon\Q_p(X)\rightarrow\Q_p(X)$ be a group homomorphism. The following statements are equivalent: 
\begin{itemize}
\item[\emph{(a)}] $A$ is $\tau$-continuous, 
\item[\emph{(b)}] the graph 
$\mathcal{G}(A):=\{(\xi, A\xi)\in\Q_p(X)\times\Q_p(X);\xi\in\Q_p(X)\} $
of the map $A$ is closed in $\Q_p(X)\times\Q_p(X)$,
\item[\emph{(c)}] for every sequence $(\xi_n)_{n\in\N}$ with $\tau\text{-}\lim_{n\rightarrow\infty}\xi_n=0$ and $\tau\text{-}\lim_{n\rightarrow\infty}A\xi_n=\eta$, we have $\eta=0$. 
\end{itemize}
\end{lemma}

Recall that a map $A\in \B(\Q_p(X))$ is called \emph{open} if it maps open sets onto open sets. As a consequence of Lemma \ref{closed graph theorem}, one easily sees that any surjective $A\in \B(\Q_p(X))$ is open. 

\subsection{The pairing on $\Q_p(X)$ and duality aspects}\label{pairing}

 We want to introduce a natural pairing on our space $\Q_p(X)$ that can be compared to a scalar product on a usual Hilbert space: Define 
$$\langle\ , \ \rangle\colon\Q_p(X)\times\Q_p(X)\rightarrow S^1$$
by 
\begin{align*}
\langle\xi, \eta\rangle:=\iota\left(\sum_{i\in X}(\xi(i)\eta(i)+\Z_p)\right) 
\end{align*}
where $\iota\colon\Q_p/\Z_p=\Z[1/p]/\Z\hookrightarrow \mathbb R/\mathbb Z\cong S^1$ is the canonical map (here, the identification $\R/\Z\cong S^1$ is given as usual by $t\mapsto e^{2\pi it}, t\in\R/\Z,$ and the identification $\Z[1/p]/\Z$ with $\Q_p/\Z_p$ is given by the composited map $\Z[1/p]\hookrightarrow\Q_p\twoheadrightarrow\Q_p/\Z_p$ that factors through $\Z[1/p]/\Z$).

The pairing is symmetric and jointly continuous because it is the composition of two continuous maps (where $\Q_p/\Z_p$ is equipped with the discrete topology). Furthermore, it induces a $\Z_p$-linear identification of $\Q_p(X)$ with its Pontryagin dual (cf. \cite{rajagopalansoundararajan} or \cite{neukirchschmidtwingberg}, Prop. I.1.1.13). As a topological group, $\Q_p(X)$ is isomorphic to its Pontryagin dual. This is an analogy to Riesz' theorem on the self-duality for Hilbert spaces.\footnote{Notice, however, that $\tau$ is \emph{not} equal to the weak topology with respect to this pairing, i.\,e. the initial topology with respect to the maps of the form $\Q_p(X)\rightarrow S^1, \xi\mapsto\langle\xi, \eta\rangle$ (where $\eta\in\Q_p(X)$), cf. the remark following Thm. 1.8.2 in W. Rudin's book ``Fourier analysis on groups'', p. 30. }\\
\begin{remark}
In the definition of the pairing $\langle\ , \ \rangle$, it would have been possible to take other embeddings $j$ of $\Q_p/\Z_p\cong\Z[1/p]/\Z$ into $S^1$ instead of $\iota$. Each such embedding differs from $\iota$ by a unique $\alpha_j\in\Z_p^\times$, the unit group of $\Z_p$, in the way that $j=\iota\circ M_{\alpha_j}$, where $M_{\alpha_j}\colon\Q_p/\Z_p\rightarrow\Q_p/\Z_p$ denotes the multiplication by $\alpha_j$.
\end{remark}
Let us pursue the analogy between $\Q_p(X)$ and Hilbert spaces: 

\begin{definition}\label{annihilator and orthogonality}
Let $K$ be a subset of $\Q_p(X)$. Define 
\begin{align*}
K^\perp:=\{\xi\in\Q_p(X);\forall\eta\in K\colon\langle\xi, \eta\rangle=0\}. 
\end{align*}
For subsets $K, L\subseteq\Q_p(X)$, we write $K\perp L$ if $\langle\xi, \eta\rangle=0$ for all $\xi\in K$ and $\eta\in L$. 
\end{definition}

 The set $K^\perp$ is a closed sub-$\Z_p$-module of $\Q_p(X)$. It may happen that $K\cap K^\perp\neq\{0\}$. For example, we have $\Z_p(X)^\perp=\Z_p(X)$. 

\begin{lemma}\label{lemma on Pontryagin duality}
Let $H\subseteq\Q_p(X)$ be a closed subgroup. The Pontryagin dual $\widehat H$ of $H$ is topologically isomorphic to $\Q_p(X)/H^\perp$. 
\end{lemma}

\begin{proof} According to Corollary 3.6.2 in \cite{deitmarechterhoff}, each character $\varphi$ on $H$ extends to $\Q_p(X)$. Because of the self-duality of $\Q_p(X)$, it can be represented by some vector $\xi\in\Q_p(X)$, i.\,e. 
\begin{align*}
\forall\eta\in H\colon \varphi(\eta)=\langle\eta, \xi\rangle 
\end{align*}
where $\xi$ is determined uniquely up to elements in $H^\perp$. We obtain a bijective group homomorphism $\Phi$ from $\widehat H$ to $\Q_p(X)/H^\perp$. Note that both groups are Polish: the second as a quotient of the Polish group $\Q_p(X)$, the first as a quotient of the Polish group $\widehat{\Q_p(X)}\cong\Q_p(X)$ (as $H$ is a closed subgroup of $\Q_p(X)$, use proposition 3.6.1 in \cite{deitmarechterhoff}). As $\Phi$ is a bijective continuous group homomorphism between Polish groups, the map $\Phi^{-1}$ must be an isomorphism of topological groups (use again Theorem 1.5 in \cite{hofmannmorris}). 
\end{proof}

\begin{lemma}\label{properties of the complement}
Let $K, L\subseteq\Q_p(X)$ be closed sub-$\Z_p$-modules. Then, the following properties hold: 
\begin{itemize}
 \item[\emph{(a)}] $K=K^{\perp\perp}$, 
 \item[\emph{(b)}] $K\subseteq L\Rightarrow L^\perp\subseteq K^\perp$, 
 \item[\emph{(c)}] $(K+L)^\perp=K^\perp\cap L^\perp$, 
 \item[\emph{(d)}] $(K\cap L)^\perp=\operatorname{cl}(K^\perp+L^\perp)$ where the closure is taken in the $\tau$-topology. 
\end{itemize}
\end{lemma}

\begin{proof} First, we prove the property (a). Lemma \ref{lemma on Pontryagin duality} yields the following exact sequence: 
\begin{align*}
0\rightarrow K^\perp\rightarrow\Q_p(X)\rightarrow\widehat K\rightarrow 0. 
\end{align*}
Now, Pontryagin duality shows (cf. \cite{deitmarechterhoff}, Corollary 3.6.2) that the dual sequence 
\begin{align*}
0\rightarrow K\rightarrow\Q_p(X)\rightarrow\widehat{K^\perp}\rightarrow 0 
\end{align*}
is also exact. Replacing $K$ by $K^\perp$ in the first sequence, we obtain the exact sequence 
\begin{align*}
0\rightarrow K^{\perp\perp}\rightarrow\Q_p(X)\rightarrow\widehat{K^\perp}\rightarrow 0. 
\end{align*}
The maps $\Q_p(X)\rightarrow\widehat{K^\perp}$ in the second and the third sequence coincide, i.\,e. their kernels $K$ and $K^{\perp\perp}$ coincide as well. The statements (b) and (c) are obvious. The statement (d) follows from statement (c) using statement (a). 
\end{proof}

\subsection{The adjoint of an operator}\label{adjoint}

 We obtain a further analogy of $\Q_p(X)$ and ordinary Hilbert spaces: 

\begin{lemma}\label{adjoint operator}
For every $A\in \B(\Q_p(X))$, there is a unique operator $A^*\in \B(\Q_p(X))$ satisfying 
$\forall\xi, \eta\in\Q_p(X)\colon \langle A\xi, \eta\rangle=\langle\xi, M\eta\rangle.$ We will call it the \emph{adjoint operator} of $A$. For $A, B\in \B(\Q_p(X))$, $\lambda \in \mathbb Z_p$ we have 
\begin{align*}
A^{**}=A, \ \ \ (A+\lambda B)^*=A^*+\lambda B^*, \ \ \ (AB)^*=B^*A^*. 
\end{align*}
\end{lemma}

\begin{proof} The uniqueness and existence of a group homomorphism $A^*$ with the above property can be proved as in the usual Hilbert space case, 
and to prove the continuity of the homomorphism $M$, one then applies the third characterization of $\tau$-continuity in Lemma \ref{closed graph theorem} (or one simply uses the fact that Pontryagin duality is a functor together with the self-duality of $\Q_p(X)$). 
 The formulae for the adjoint operator are clear. 
\end{proof}

\begin{lemma}\label{adjoint operator and its image}
For every $A\in \B(\Q_p(X))$, we have $\operatorname{ker}(A)=\operatorname{im}(A^*)^\perp$. 
\end{lemma}

\begin{proof} The direction $\operatorname{ker}(A)\subseteq\operatorname{im}(A^*)^\perp$ is clear. Suppose therefore $\eta\in\operatorname{im}(A^*)^\perp$. For all $\xi\in\Q_p(X)$, we see that 
\begin{align*}
\langle A\eta, \xi\rangle=\langle\eta, A^*\xi\rangle=0. 
\end{align*}
Hence, since our natural pairing is non-degenerate, we obtain that $A\eta=0$ or $\eta\in\operatorname{ker}(A)$. 
\end{proof}

For reasons of completeness, we finally want to state a more general version of Lemma \ref{adjoint operator}: 

\begin{theorem}\label{sesquilinear form and representation}
Let $\sigma\colon\Q_p(X)\times\Q_p(X)\rightarrow S^1$ be a biadditive form that is separately continuous. Then, there exists a unique $A\in\B(\Q_p(X))$ such that 
\begin{align*}
\forall\xi, \eta\in\Q_p(X)\colon\langle A\xi, \eta\rangle=\sigma(\xi, \eta) 
\end{align*}
holds\footnote{Note that the seemingly non-trivial part lies in showing that already \emph{separate} continuity of $\sigma$ is sufficient. }. 
\end{theorem}

The proof can be found in \cite{claussnitzer}, Theorem 4.4. 

\subsection{The norm topology on $\Q_p(X)$ and $\B(\Q_p(X))$}\label{norm topology}

 For an element $\xi\in\Q_p(X)$, we define $\|\xi\|:=\max_{i\in X}|\xi(i)|_p$. It is clear that we have defined an ultra-norm on the $\Z_p$-module $\Q_p(X)$ in this way: $\|\xi+\eta\|\leq \max\{ \|\xi\|, \|\eta\| \}$ for all $\xi, \eta\in\Q_p(X)$. Note that all norm-convergent sequences also converge with respect to $\tau$, but not the other way around. The norm topology is therefore stronger than the $\tau$-topology (strictly stronger if $X$ is infinite). The following lemma is easy to verify: 

\begin{lemma}\label{compactness lemma}
A subset $K\subseteq\Q_p(X)$ is $\tau$-compact if and only if 
it is norm-bounded, $\tau$-closed and 
 there is a finite subset $S\subseteq X$ such that 
$$K\subseteq\prod_{x\in S}\Q_p\times\prod_{x\in X\setminus S}\Z_p. $$
\end{lemma}

 We want to investigate some further properties of the norm and of norm-continuous operators. The following two lemmas are easy to verify: 

\begin{lemma}\label{norm-completeness of our space}
The space $\Q_p(X)$ is complete with respect to the norm. 
\end{lemma}


\begin{lemma}\label{boundedness and continuity}
Let $A\colon \Q_p(X)\rightarrow\Q_p(X)$ be a $\Z_p$-linear map. Then, $A$ is norm-continuous if and only if $A$ is bounded, i.\,e. there is $C>0$ such that 
\begin{align*}
\forall\xi\in\Q_p(X)\colon \|A\xi\|\leq C\|\xi\|. 
\end{align*}
\end{lemma}


\begin{lemma}\label{continuity in different topologies}
A $\tau$-continuous $\Z_p$-linear map on $\Q_p(X)$ is also norm-continuous. 
\end{lemma}

\begin{proof} Suppose that $A$ is a $\tau$-continuous $\Z_p$-linear map, i.\,e. that $A\in \B(\Q_p(X))$. As $\Z_p(X)$ is $\tau$-compact, also its image under $A$ is $\tau$-compact and therefore norm-bounded by Lemma \ref{compactness lemma}. This fact implies the boundedness and hence the continuity of $A$. 
\end{proof}

Unfortunately, the converse does not hold (this is a consequence for example of Theorem 2.4.1 in \cite{claussnitzer}). 

\begin{definition}\label{operator norm}
For each $A\in \B(\Q_p(X))$, we define its \emph{operator norm} in the usual way by 
\begin{align*}
\|A\|:=\sup_{\xi\in\Q_p(X), \|\xi\|\leq1}\|A\xi\|. 
\end{align*}
\end{definition}

 By Lemma \ref{continuity in different topologies}, this is a real number and it is clear that it makes $\B(\Q_p(X))$ an ultra-normed $\Z_p$-module. For $A, B\in \B(\Q_p(X))$ and $\xi\in\Q_p(X)$, we have 
\begin{align*}
\|A+B\|\leq \max\{ \|A\|, \|B\|\} \ \ \|AB\|\leq\|A\|\|B\|, \ \ \|A\xi\|\leq\|A\|\|\xi\|. 
\end{align*}

\begin{lemma}\label{completeness of the operators}
The $\Z_p$-module $\B(\Q_p(X))$ is norm-complete. 
\end{lemma}

 Once we will have established the matrix representation of the operators in $\B(\Q_p(X))$ (Theorem \ref{matrix representation of the operators}), this lemma will be easy to show, and therefore we skip the proof for the moment. 

\subsection{Mahler's algebra and continuous functional calculus}\label{functional calculus}

 For $x\in\Z_p$ and $n\in\N$, we will need the binomial coefficient 
\begin{align*}
\binom{x}{k}:=\frac{x(x-1)\ldots(x-(n-1))}{n!} \in \mathbb Z_p. 
\end{align*}
The next lemma has a nice combinatorial proof.

\begin{lemma}\label{nice combinatorial identity}
\begin{itemize}
 \item[\emph{(a)}] For $x\in\Z_p$ and $m, n\in\N$, the following identity holds: 
\begin{align*}
\binom{x}{m}\binom{x}{n}=\sum_{l=m\vee n}^{m+n}\frac{l!}{(m+n-l)!(l-m)!(l-n)!}\binom{x}{l}. 
\end{align*}
\item[\emph{(b)}] For $x\in\Z_p$ and $n\in\N$, the following identity holds: 
\begin{align*}
x\binom{x}{n}=n\binom{x}{n}+(n+1)\binom{x}{n+1}. 
\end{align*}
\end{itemize}
\end{lemma}

\begin{proof} (a) It is sufficient to show the formula for the case $x\in\N$, $x>m+n$. We assume this.

Then, consider a finite set $X$ with cardinality $|X|=x$. The left side of the above equation is exactly the number of pairs $(M, N)$ of subsets $M, N\subseteq X$ such that $|M|=m$ and $|N|=n$. Each such pair is uniquely characterized by the set $M\cup N$ and the subdivision of $M\cup N$ into the subsets $M\setminus N$, $N\setminus M$ and $M\cap N$ and this is precisely what the right side corresponds to:  Indeed, the number $l$ corresponds to $|M\cup N|$, the binomial coefficient on the right corresponds to the choices of the set $M\cup N$ and the fraction to the number of subdivisions. Hence, the two sides of the equation coincide.

(b) This is just a consequence of the first part of the lemma (set $m=1$). 
\end{proof}

The following theorem is due to Mahler, see \cite{bojanic} for an elementary proof.

\begin{theorem}[Mahler's theorem]\label{Mahler's theorem}
Every element $f\in C(\Z_p, \Z_p)$ has a unique representation of the form 
\begin{align*}
f(x)=\sum_{n=0}^\infty T_n(f)\binom{x}{n} 
\end{align*}
such that  $T_n(f)\in\Z_p$ and $\lim_{n\rightarrow\infty}T_n(f)=0.$ 
The convergence of this series is uniform and the equality 
\begin{align*}
\|f\|_{\text{\emph{sup}}}=\max_{n\in\N}|T_n(f)|_p 
\end{align*}
holds. In other words, there is an isometric isomorphism $\sigma\colon C(\Z_p, \Z_p)\rightarrow c_0(\N, \Z_p)$ of $\Z_p$-modules given by $f\mapsto(T_n(f))_{n\in\N}$. 
\end{theorem}

\begin{definition}\label{normal contraction}
An operator $A\in \B(\Q_p(X))$ is called a \emph{normal contraction} if the quotient 
\begin{align*}
\binom{A}{n}:=\frac{A(A-1)\ldots(A-(n-1))}{n!}\in \B(\Q_p(X)) 
\end{align*}
is defined and is a contraction, i.\,e. its norm is not greater than one. 
\end{definition}

 It is not difficult to show that $|n!|_p=p^{-\frac{n-s_p(n)}{p-1}}$ for $n\in\N$ where $s_p(n)$ denotes the digit sum in the $p$-adic decomposition 
\begin{align*}
n=\sum_{k=0}^\infty n_kp^k 
\end{align*}
of $n$ (with $n_k\in\{0, \ldots, p-1\}$), i.\,e. 
\begin{align*}
s_p(n)=\sum_{k=0}^\infty n_k. 
\end{align*}
Therefore, we obtain that $A$ is a normal contraction if and only if 
\begin{align*}
\forall n\in\N\colon\|A(A-1)\ldots(A-(n-1))\|\leq p^{-\frac{n-s_p(n)}{p-1}}. 
\end{align*}
For example, a contractive diagonal operator on $\Q_p(X)$ is always a normal contraction.
Note that the formulae in Lemma \ref{nice combinatorial identity} remain true if one replaces $x$ by a normal contraction $A$.  If $A$ is a normal contraction, we obtain a natural functional calculus using Mahler's theorem: 

\begin{theorem}\label{Mahler's functional calculus}
If $A\in \B(\Q_p(X))$ is a normal contraction, then there is a natural contractive homomorphism of $\Z_p$-algebras 
\begin{align*}
\pi_A\colon C(\Z_p, \Z_p)\rightarrow \B(\Q_p(X)) 
\end{align*}
with $\pi_A(\operatorname{id}_{\Z_p})=A$. 
\end{theorem}

 As usual, we write $f(A)$ instead of $\pi_A(f)$.
 Note that for a normal contraction $A$ and $f\in C(\Z_p, \Z_p)$, also the operator $f(A)$ is a normal contraction because as $f(A)$ can be represented by a function in $C(\Z_p, \Z_p)$, also the binomial coefficients $\tbinom{f(A)}{n}$ can and are therefore well-defined contractions. 
 
\begin{proof} By Theorem \ref{Mahler's theorem}, there is a natural isometric isomorphism of $\Z_p$-modules $\sigma\colon C(\Z_p, \Z_p)\rightarrow c_0(\N, \Z_p)$ satisfying $\sigma(\operatorname{id}_{\Z_p})=\delta_1$. For $f\in C(\Z_p, \Z_p)$, define 
\begin{align*}
\pi_A(f):=\sum_{n=0}^\infty\sigma(f)(n)\binom{A}{n}. 
\end{align*}
This definition yields a contractive homomorphism of $\Z_p$-algebras and the proof is finished. 
\end{proof}

 For example, if $A\in \B(\Q_p(X))$ is a normal contraction and $z\in p\Z_p$, the operator 
\begin{align*}
F_z(A):=\sum_{n=0}^\infty z^n\binom{A-1}{n} 
\end{align*}
is well-defined.

\begin{example} An example of a normal contraction $A$ acting on the space $\Q_p(\N)$ is given by the operator defined by $A(\delta_n)=n\delta_n+(n+1)\delta_{n+1}$. Indeed, one can show by induction that the $n$-th row of the matrix representing the operator $A(A-1)\ldots(A-k)$ (cf. Theorem \ref{matrix representation of the operators}) is given by 
\begin{align*}
\left(\binom{k+1}{n-i}\binom{n}{k+1}(k+1)!\right)_{i\in\N} 
\end{align*}
for $n, k\in\N$. 
\end{example}

Let's recall the following lemma.

\begin{lemma}\label{number theory and convergence}
The sequence $(f_n)$ of functions $\Z_p\rightarrow\Z_p$ that is defined by 
\begin{align*}
f_n(x)=x^{p^n} 
\end{align*}
for all $x\in\Z_p$ converges uniformly to a function that is constant on each equivalence class for the equivalence relation of having distance less than $1$. 
\end{lemma}

 The result is well known and the limit $\lim_{n\rightarrow\infty}f_n(x)$ is called the \emph{Teichmüller representative} of $x$ (cf. \cite{maninpanchishkin}, Chapter 4.3.4). The proof is also repeated in \cite{claussnitzer}, Lemma 1.6.5. \\

 Now, it is possible to define a polynomial with coefficients in $\Z_p$ mapping all the non-zero Teichmüller representatives to $0$ and $0$ to $1$, namely the polynomial $P_{\Q_p}(X):=\frac{(X-\lambda_1)\ldots(X-\lambda_{p-1})}{(-1)^{p-1}\lambda_1\ldots\lambda_{p-1}}$. 

\begin{corollary}\label{normal contractions and number theory}
Suppose that $A\in\B(\Q_p(X))$ is a normal contraction. Then, the sequence $P_{\Q_p}(A^{p^n})$ converges to an idempotent in the operator norm. 
\end{corollary}

 \emph{Remark:} It is also possible to formulate the above functional calculus for finite field extensions $K$ of $\Q_p$ (cf. \cite{claussnitzer}, Chapter 1.6), but we prefered working with $\Q_p$ for now.

\subsection{The matrix representation of operators}\label{matrix representation}

 Let $A$ be an operator in $\B(\Q_p(X))$. Associate the matrix $M_A:=(A_{ij})_{i, j\in X}$ to $A$ whose coefficients are given by $A_{ij}=(A(\delta_j))(i)$. Note that $A\in\B(\Q_p(X))$ is uniquely determined by $M_A$. Furthermore, for continuity reasons, we have $A(\xi)(i)=\sum_{j\in X}A_{ij}\xi(j)$ for all $\xi\in\Q_p(X)$.
 
First, we will state a lemma and second, we will characterize all matrices that can be written in the form $M_A$ for an operator $A\in \B(\Q_p(X))$. 

\begin{lemma}\label{adjoint matrix}
Let $A$ be in $\B(\Q_p(X))$, then we have $M_{A^*}=M_A^T$, where $M_A^T=(A_{ji})_{i, j\in X}$ is just the transposed matrix of $M_A$. 
\end{lemma}

\begin{proof} Let $\lambda$ be a number in $\Q_p$ and $i, j\in X$. Observe 
\begin{align*}
\iota(\lambda A_{ij}+\Z_p)=\langle A\delta_j, \lambda\delta_i\rangle=\langle\lambda\delta_j, A^*\delta_i\rangle=\iota(\lambda A^*_{ji}+\Z_p). 
\end{align*}
This can only hold for every $\lambda\in\Q_p$ if $A^*_{ji}=A_{ij}$ for all $i, j\in X$. Therefore, $M_{A^*}$ is exactly the transpose of $M_A$. 
\end{proof}

\begin{theorem}\label{matrix representation of the operators}
A necessary and sufficient condition for a matrix $M=(a_{ij})_{i, j\in X}$ to be of the form $M=M_A$ for an operator $A\in \B(\mathbb Q_p(X))$ is that \emph{(a)} $M$ admits only finitely many entries in $\mathbb Q_p\setminus\mathbb Z_p$ and that \emph{(b)} for $k\in X$ one always has $\lim_{i\rightarrow\infty}a_{ik}=0$ and $\lim_{j\rightarrow\infty}a_{kj}=0$. 
\end{theorem}

\begin{proof} To see why (a) is necessary, suppose that $M$ has infinitely many entries in $\mathbb Q_p\setminus\mathbb Z_p$. As in each row and in each column there are clearly only finitely many entries in $\mathbb Q_p\setminus \mathbb Z_p$, it is possible to choose an infinite subset $Y\subseteq X$ such that for each $y\in Y$ one has $\{i\in X; a_{iy}\in\mathbb Q_p\setminus\mathbb Z_p\}\neq\varnothing$ and $\{i\in X; a_{iy}\in\mathbb Q_p\setminus\mathbb Z_p\}\cap\{j\in X; a_{jz}\in\mathbb Q_p\setminus\mathbb Z_p\}=\varnothing$ for $y, z\in Y, y\neq z$. Consider the element $\chi_Y\in\mathbb Q_p(X)$, the characteristic function of the set $Y$. One has $\chi_Y=\lim_{n\rightarrow\infty}\chi_{Y_n}$ (convergence with respect to $\tau$) where $(Y_n)_{n\in\mathbb N}$ is an increasing sequence of finite subsets of $Y$ with the property that $Y=\bigcup_n Y_n$. If there existed $A\in \B(\mathbb Q_p(X))$ such that $M=M_A$, the sequence $(A\chi_{Y_n})_n$ would by continuity converge in $\mathbb Q_p(X)$. The choice of the set $Y$ shows that this is not the case. Therefore, condition (a) is necessary for the existence of such an operator $A$.

 On the other hand, suppose that there is an element $x\in X$ and $\varepsilon>0$ such that $\{j\in X; |a_{jx}|>\varepsilon\}$ is infinite. For $\lambda\in\mathbb Q_p$ with $\varepsilon|\lambda|>1$, the element $\lambda\chi_{\{x\}}$ lies in $\mathbb Q_p(X)$, but as $(\lambda a_{ix})_{i\in X}$ does not lie in $\mathbb Q_p(X)$, the matrix $M$ is not of the form $M=M_A$ for $A\in\B(\mathbb Q_p(X))$. The same holds for the case that $\{j\in X; |a_{xj}|>\varepsilon\}$ is infinite (considering the adjoint matrix $M^*$ and using the lemma above). Therefore, condition (b) is equally necessary for the existence of such an $A\in\B(\mathbb Q_p(X))$.
 
 In order to prove that (a) and (b) are sufficient for the existence of $A$, define $A$, being given a matrix $M$ such that (a) and (b) hold, by $(A\xi)_i=\sum_{j\in X}a_{ij}\xi_j$ where $\xi=(\xi_j)_{j\in X}\in \mathbb Q_p(X), i\in X$. One can easily verify that $A$ lies indeed in $\B(\mathbb Q_p(X))$ and that $M=M_A$. 
\end{proof}

 The following lemma is easy to prove: 

\begin{lemma}\label{norm of a matrix}
Let $A$ be in $\B(\Q_p(X))$, then we have $\|A\|=\max\{|A_{ij}|_p; i, j\in X\}$. 
\end{lemma}


 Using Theorem \ref{matrix representation of the operators} and Lemma \ref{norm of a matrix}, Lemma \ref{completeness of the operators}, i.\,e. the completeness of $\B(\Q_p(X))$ with respect to the norm becomes obvious. 

 To finish the section, we want to give a link of our topic to Willis' notion of the scale of an operator. Recall that, for an endomorphism $\alpha$ on a totally disconnected locally compact group $G$ (i.\,e. a continuous group homomorphism $G\rightarrow G$), the \emph{scale} $s(\alpha)$ is defined as the minimum of all possible values $[\alpha(U):(U\cap\alpha(U))]$ for compact open subgroups $U$ of $G$ (the group $G$ always has a base of neighborhoods of the identity that consists only of compact open subgroups, cf. Theorem 2.1 in \cite{willis}). 
 
 For an arbitrary compact open subgroup $U$ of $G$, the scale can be calculated as $s(\alpha)=\lim_{n\rightarrow\infty}[\alpha^n(U):(U\cap\alpha^n(U))]^{1/n}$, cf. Proposition 8.3 in \cite{willis}. 
 
 Also for operators in $\mathcal B(\mathbb Q_p(X))$, we can ask how to calculate their scale. If $X$ is finite, then we have $\mathbb Q_p(X)=\mathbb Q_p^n$ (for an appropriate $n\in\mathbb N$) and $s(\alpha)$ is the norm of the product of all eigenvalues of $\alpha$ with norm greater than $1$ (in a finite field extension of $\Q_p$, where the characteristic polynomial of $\alpha$ decomposes in linear factors; use the Frobenius normal form to show this), i.\,e.
 $$s(\alpha)= \sup_n \left\|\bigwedge \!\! {}^n\alpha \right\|.$$ However, it seems to be a more difficult question how to determine the scale of an operator in $\mathcal B(\mathbb Q_p(X))$ for infinite $X$. It seems reasonable to expect that the scale of a general operator is the limit of the scales of the finite minors in its matrix representation and that a similar formula as above holds -- but we were unable to show this.
 

 For every operator $A\in \B(\Q_p(X))$, we have $s(A^*)=s(A)$. Even a more general statement can be proved: Let $G$ be a totally disconnected locally compact \emph{abelian} group and $A$ an endomorphism on $G$; then, the adjoint endomorphism $A^*$ acting on the Pontryagin dual $G'$ of $G$ has the same scale as $A$.


\section{Various operator algebras and their $K$-theory}
\subsection{Compact operators in $\B(\mathbb Q_p(X))$}\label{compact operators in our context}

 It is interesting to see that also the ideal of compact operators of usual Archimedean functional analysis have a natural analogy in our context: 

\begin{definition}\label{compact operators}
Define $\K(\mathbb Q_p(X))$ to be the set of all operators in $\B(\mathbb Q_p(X))$ that map norm-bounded sets onto relatively $\tau$-compact sets in $\Q_p(X)$. We want to call the elements of $\K(\mathbb Q_p(X))$ the \emph{compact operators} on $\mathbb Q_p(X)$. 
\end{definition}

 In the rest of this section, we will always assume $X=\N$ (without any restriction of generality). 

\begin{lemma}\label{compact operators and the norm-topology}
For an operator $A\in \B(\Q_p(\N))$, the following three statements are equivalent: 
\begin{itemize}
 \item[\emph{(a)}] $A$ is a compact operator, 
 \item[\emph{(b)}] the matrix-entries of $A$ converge to zero, 
 \item[\emph{(c)}] it maps norm-bounded sets onto relatively norm-compact sets in $\Q_p(\N)$. 
\end{itemize}
 The operators with this property form a self-adjoint ideal in $\B(\Q_p(\N))$, i.\,e. an ideal that is closed under the adjoint operation. 
\end{lemma}

\begin{proof} (a)$\Rightarrow$(b): Consider $N\in\N$. If $A$ is compact, then the image of $M:=\{p^{-N}\delta_n; n\in\N\}$ (as a norm-bounded set) must be relatively $\tau$-compact. According to Lemma \ref{compactness lemma}, the entries of the elements of $A(M)$ have always to be in $\Z_p$ for sufficiently high indices. But the entries of $A(p^{-N}\delta_n)$ are exactly the matrix entries of the $n$-th column of $A$, multiplied by $p^{-N}$. This shows that the matrix entries of $A$ must have norm at most $p^{-N}$ for sufficiently high row-numbers. But in the (only finitely many) rows where entry-norms greater than $p^{-N}$ occur, $A$ can only have finitely many entries with norm greater than $p^{-N}$ because the row entries converge to zero in each row (cf. Theorem \ref{matrix representation of the operators}). Therefore, $A$ has only finitely many matrix entries of norm greater than $p^{-N}$.
 As $N\in\N$ is arbitrary, the matrix entries of $A$ must converge to zero.
 
(b)$\Rightarrow$(c): Suppose that the matrix entries of $A$ converge to zero. To prove (c), it is sufficient to show that all sets of the form $M_N:=A(\{\xi\in\Q_p(\N);\|\xi\|\leq p^N\}), N\in\N$, are relatively compact in $\Q_p(\N)$. Suppose $N\in \N$. There exists, for each $k\in\N$, a number $m_k\in\N$ such that all matrix entries of $A$ in a row with number $n\geq m_k$ have norm less than $p^{-N}p^{-k}$. We now obtain $|(A\xi)(n)|\leq p^{-k}$ for $\xi\in\Q_p(\N)$ with $\|\xi\|\leq p^N$ and $n\geq m_k$. Therefore, we can construct a sequence $(p^{-a_l})_{l\in\N}$ ($a_l\in\Z$ for $l\in\N$) with $p^{-a_l}\xrightarrow{l\rightarrow\infty}0$ (in $\R$) such that $|(A\xi)(n)|\leq p^{-a_n}$ for all $n\in\N$ and $\xi\in\Q_p(\N), \|\xi\|\leq p^N$. We see that 
\begin{align*}
M_N\subseteq Q:=\prod_{l\in\N}B_{p^{-a_l}} 
\end{align*}
where $B_\varepsilon:=\{\lambda\in\Q_p; |\lambda|\leq\varepsilon\}, \varepsilon>0$. To prove (c), it is sufficient to show the norm-compactness of $Q$. But this is a consequence of the Tychonoff theorem: Notice that the norm-topology on $Q$ coincides exactly with the product topology because we assumed $p^{-a_l}\xrightarrow{l\rightarrow\infty}0$ (in $\R$).

(c)$\Rightarrow$(a): This is clear since every relatively norm-compact set in $\Q_p(\N)$ is also relatively $\tau$-compact (note that a norm-convergent sequence in $\Q_p(\N)$ is also $\tau$-convergent).

The fact that the compact operators form a self-adjoint ideal in $\B(\Q_p(\N))$ follows easily if one uses the matrix representation for compact operators. 
\end{proof}

\subsection{Some results on idempotents in $\B(\Q_p(X))$}\label{some results on idempotents}


 In the following sections, we want to analyze properties of idempotents in $\B(\Q_p(X))$ and calculate the $K_0$-groups of $\K(\Q_p(X))$ and $\B(\Q_p(X))$. As a good introduction into $K$-theory, we recommend \cite{rosenberg}.
 
The $K$-theory of nonarchimedean Banach rings (i.\,e. complete normed rings whose norm satisfies submultiplicativity and the strong triangle inequality) has been investigated by Adina Calvo in her thesis \cite{calvo}.
 
We want to collect first information on the idempotents in $\B(\Q_p(X))$. If $A\in \B(\Q_p(X))$ is an \emph{idempotent}, i.\,e. it satisfies the equation $A^2=A$, then the operators $1-A$, $A^*$ and $1-A^*$ are idempotents as well. Note that we have $\operatorname{ker}(1-A)=\operatorname{im}(A)$ and that similar equations hold for $A^*$, $1-A$ and $1-A^*$ instead of $A$. A self-adjoint idempotent will be called a \emph{projection}. The following lemma is easy to prove: 

\begin{lemma}\label{image and kernel of an idempotent}
For an idempotent $A\in \B(\Q_p(X))$, we have the following identities: 
\begin{align*}
\operatorname{im}(A)^\perp=\operatorname{im}(1-A^*)\ \text{ and }\ \operatorname{im}(A)=\operatorname{im}(1-A^*)^\perp. 
\end{align*}
\end{lemma}

 Note that $\operatorname{im}A=\operatorname{ker}(1-A)$ is closed. Combining the preceding lemma with Lemma \ref{lemma on Pontryagin duality}, we obtain the following: 

\begin{lemma}\label{dual of the image of an idempotent}
If $A\in \B(\Q_p(X))$ is an idempotent, then the Pontryagin dual of $\operatorname{im}(A)$ is isomorphic to $\operatorname{im}(A^*)$. 
\end{lemma}

It would be interesting to know if one can define the usual operations (like supremum and infimum) on the set of idempotents (or projections) in our context.
 
Our first conjecture in this direction was that for a sequence $(e_n)_{n\in\N}$ of idempotents in $\B(\Q_p(\N))$ with 
\begin{align*}
\forall n\in\N\colon &e_{n+1}\Q_p(\N)\subseteq e_n\Q_p(\N), \\
&(1-e_n)\Q_p(\N)\subseteq(1-e_{n+1})\Q_p(\N), \\
&\|e_n\|\leq1, 
\end{align*}
there always exists an idempotent $e\in \B(\Q_p(\N))$ such that 
\begin{align*}
e\Q_p(\N)&=\bigcap_{n\in\N}e_n\Q_p(\N), \\
(1-e)\Q_p(\N)&=\tau\text{-cl}\left(\bigcup_{n\in\N}(1-e_n)\Q_p(\N)\right). 
\end{align*}
This conjecture, however, turns out to be false (even in the case that the $e_n$ are required to be projections). Counter-examples can be found in \cite{claussnitzer}, Section 3.1. \\
 Second, we would like to know if for two projections $e,f\in \B(\Q_p(\N))$, there is always a projection (or at least an idempotent) $g\in \B(\Q_p(\N))$ such that $\operatorname{im}g=\operatorname{im}e\cap\operatorname{im}f$. Unfortunately, also this conjecture is false (cf. Section 3.1 in \cite{claussnitzer} for a counter-example). \\

\begin{theorem}\label{intersection counter-example}
There is a decreasing sequence of contractive projections $(e_n)_{n\in\N}$ in $\B(\Q_p(X))$ that is decreasing such that $\bigcap_{n\in\N}e_n\Q_p(\N)$ is not the image of an idempotent in $\B(\Q_p(X))$. There are contractive projections $e, f\in \B(\Q_p(\N))$ such that $\operatorname{im}e\cap\operatorname{im}f$ is not the image of an idempotent in $\B(\Q_p(\N))$. 
\end{theorem}


\subsection{The group $K_0(\K(\Q_p(X)))$}\label{compact K-group}

 In order to calculate $K_0(\K(\mathbb Q_p(X)))$, we will first establish some more general lemmas. 
 
  Two idempotents $e, f \in A$ in a unital Banach-$\mathbb Z_p$-algebra $A$ are called \emph{equivalent} with respect to $A$ if there is an invertible element $g\in A$ such that $g^{-1}eg=f$. The following two lemmas should essentially be well-known and in fact holds for an arbitrary unital Banach-$\Z_p$-algebra. 

\begin{lemma}\label{close projections}
Let $A$ be a closed sub-$\mathbb Z_p$-algebra of $\B(\mathbb Q_p(X))$ that contains the identity. Let $e, f \in A$ be idempotents such that $e\neq 0$ and $\|e-f\|<1/\|e\|$. Then, $e$ and $f$ are equivalent with respect to $A$. 
\end{lemma}

\begin{proof} If $e$ and $f$ are as in the lemma, we obtain 
\begin{align*}
\|f+e-2fe\|&=\|f-fe+e-fe\|\\
&\leq\max \{\|e\|\|e-f\|, \|f\|\|e-f\|\}<\|e\|\frac{1}{\|e\|}=1 
\end{align*}
because $\|e\|\geq 1>\|e-f\|$ and therefore $\|f\|=\|e\|$. As in the Archimedean case, one can, $A$ being closed, use the Neumann series (geometric series) to show that the element $u=1-f-e+2fe\in A$ is invertible in $A$. On the other hand, one has $fu=fe=ue$ and the lemma follows. 
\end{proof}

\begin{lemma}\label{general lemma on approximation} 
Let $\mathcal A$ be an ultra-normed Banach algebra. Suppose that $a\in\mathcal A\setminus\{0\}$ satisfies $\|a^2-a\|<1/\|a\|^2$. Then, there is an idempotent element $e\in\mathcal A$ such that $\|a-e\|<\min\{1/ \|a\|,1\}$. The idempotent $e$ is given as the limit of the sequence $P_m(a)$ as $m\rightarrow\infty$ for a certain sequence $P_m$ of polynomials in $\Z[x]$. 
\end{lemma}

\begin{proof} The result is clear if $\|a\|<1$, so suppose $\|a\|\geq1$. Then, we have in particular $\|a^2-a\|<1$. 
First, we will have to establish that for each $m\in\mathbb N, m\geq 1$, there is exactly one polynomial $P_m\in\mathbb Z[x]$ such that 
\begin{align*}
P_m(0)=0, P_m(1)=1 \textnormal{ and } P_m^{(i)}(0)=P_m^{(i)}(1)=0 
\end{align*}
for $i\in\{1,\ldots,m-1\}$ and $\operatorname{deg}P_m\leq 2m-1$. The ansatz $P_m(x)=\sum_{i=0}^{2m-1}a_i x^i$ yields $P_m^{(k)}(x)=\sum_{i=k}^{2m-1}a_i(i(i-1)\ldots(i-k+1))x^{i-k}$, thus $f^{(k)}(0)=a_k k!=0$ and $a_k=0$ for $k\in\{1, \ldots, m-1\}$. Furthermore, one gets 
\begin{align*}
f^{(k)}(1)=\sum_{i=m}^{2m-1}a_i\binom{i}{k}=\delta_k, k\in\{0, \ldots, m-1\}, 
\end{align*}
where $\delta_k$ denotes the value $1$ for $k=0$ and $0$ else. The resulting system of linear equations has $m$ equations and $m$ variables. Using a result from \cite{aignerziegler}, Chapter ``Gitterwege und Determinanten'', one can easily see that the determinant of the coefficient matrix of this system is $1$. Therefore, it admits a unique solution and the unique existence of the polynomial $P_m\in\mathbb Z[x]$ is proved\footnote{It is possible to prove an explicit formula for the polynomials $P_m$: 
\begin{align*}
P_m(x)=\sum_{k=m}^{2m-1}x^k\sum_{i=k-m+1}^{m}(-1)^{i+1}\binom{m}{i}\binom{m-1+k-i}{k-i}. 
\end{align*}
}. \\
 Consider now the sequence $P_m(a)$. We notice that 
\begin{align*}
\|P_{m+1}(a)-P_m(a)\|=\|(a^2-a)^m(\alpha a+\beta)\|\leq\|a^2-a\|^m\|a\|\rightarrow0
\end{align*}
for $m\rightarrow\infty$ (where $\alpha, \beta\in\Z$) and 
\begin{align*}
\|P_m(a)^2-P_m(a)\|=\|(a^2-a)^m g_m(a)\|\leq\|a^2-a\|^m\|a\|^{2m-2} 
\end{align*}
for a certain polynomial $g_m$ of degree at most $2m-2$ over $\Z$. Choose $d>2$ such that $\|a^2-a\|<1/\|a\|^d$ and define $c:=1-2/d>0$, i.\,e. $d(1-c)=2$. Then, we obtain 
\begin{align*}
\|P_m(a)^2-P_m(a)\|\leq\|a^2-a\|^m\|a\|^{2m-2}&<\frac{1}{\|a\|^{d(1-c)m}}\|a^2-a\|^{cm}\|a\|^{2m-2}\\&=\frac{\|a^2-a\|^{cm}}{\|a\|^2}\rightarrow0 
\end{align*}
for $m\rightarrow\infty$. Hence, the sequence $(P_m(a))$ converges to an idempotent $e\in\mathcal A$. The inequality $\|P_{m+1}(a)-P_m(a)\|\leq\|a^2-a\|^m\|a\|<1/\|a\|$ for $m\in\N, m\geq1$ and the convergence $P_{m+1}(a)-P_m(a)\rightarrow0$ imply that $\|e-a\|<1/\|a\|$. 
\end{proof}

\begin{theorem}\label{continuity in the unital case} 
Let $(A_n)_{n\in\mathbb N}$ be an increasing sequence of closed sub-$\mathbb Z_p$-algebras of $\B(\mathbb Q_p(X))$. Moreover, let $A$ be the closed union of the $A_n$. Then $K_0(A)$ is isomorphic to the direct limit of the sequence of the $K_0(A_n)$ with the canonical homomorphisms. 
\end{theorem}

\begin{proof} The proof is (as in the Archimedean case) a straightforward application of the two preceding lemmas (cf. \cite{murphy}, pp. 234-240), cf. also \cite{claussnitzer}, p. 46.
\end{proof}

 Again, a more general result is true: Let $(A_n, \varphi_n)$ be a sequence of Banach-$\mathbb Z_p$-algebras $A_n$ and contractive homomorphisms $\varphi_n\colon A_n\rightarrow A_{n+1}$, and let $A$ be their direct limit as a $\mathbb Z_p$-Banach algebra. Then, $K_0(A)$ is the direct limit of the sequence $(K_0(A_n), K_0(\varphi_n))$. 

Observe that $\K(\Q_p(X))$ is the closure of the set of all operators whose matrices have only finitely many non-vanishing entries. As the finite-dimensional matrix algebras $\Q_p^{n\times n}$ (as well as $\Z_p^{n\times n}$) have $K_0$-group $\Z$, we can therefore state the following corollary as an application of the preceding theorem (here, we let $\K_{(1)}(\mathbb Q_p(X))$ denote the set of all operators in $\K(\mathbb Q_p(X))$ with norm not greater than $1$): 

\begin{corollary}\label{compact operators and K-groups}
We have $K_0(\K(\mathbb Q_p(X)))=\mathbb Z$. 

\noindent The canonical map $K_0(\K_{(1)}(\mathbb Q_p(X)))\rightarrow K_0(\K(\mathbb Q_p(X)))$ is an isomorphism. 
\end{corollary}


\begin{theorem}\label{compact idempotents}
Let $e$ be an idempotent in $\K(\mathbb Q_p(X))$. Then, the image of $e$ is a finite dimensional $\mathbb Q_p$-vector space. 
\end{theorem}

\begin{proof} A compact operator $e$ in $\B(\Q_p(X))$ has the property that it can be approximated in norm by an operator $F$ in $\B(\Q_p(X))$ having only finitely many non-vanishing matrix-entries such that $\|e-F\|<1/\|e\|^3$. If $e$ is an idempotent, then we have $\|F^2-F\|\leq \max\{\|F^2-e^2\|,\|e-F\|\}\leq \max\{ \|F-e\|\|F\|,\|F-e\|\|e\|,\|e-F\|\}<1/\|e\|^2$. Therefore, there is an idempotent $f$ with only finitely many matrix entries such that $\|f-F\|<1/\|e\|$. We also obtain $\|f-e\|<1/\|e\|$ and therefore the equivalence of $e$ and $f$. As $f$ has finite-dimensional image, also $e$ must have finite-dimensional image. 
\end{proof}

\subsection{The group $K_0(\B(\Q_p(X)))$}\label{vanishing K-group}

Let $\B_{(1)}(\mathbb Q_p(X))$ denote the set of all operators in $\B(\mathbb Q_p(X))$ with norm not greater than $1$. 
Next, we want to show that $K_0(\B_{(1)}(\mathbb Q_p(X)))=0$. 

\begin{lemma}\label{infinite sum ring}
If $X$ is countably infinite, the ring $\B_{(1)}(\mathbb Q_p(X))$ is an infinite sum ring. In particular, $K_0(\B_{(1)}(\mathbb Q_p(X)))=0$. 
\end{lemma}

\begin{proof} First, we show that it is a sum ring\footnote{ A \emph{sum ring} is a unital ring $R$ with elements $a_0, b_0, a_1, b_1\in R$ such that $a_0b_0=a_1b_1=1$ and $b_0a_0+b_1a_1=1$, cf. \cite{cortinas}, p. 10. In this case, $\boxplus\colon R\times R\rightarrow R, (x, y)\mapsto x\boxplus y=b_0xa_0+b_1ya_1$, is a unital ring homomorphism. An \emph{infinite sum ring} is a sum ring $R$ with a unital ring homomorphism $R\rightarrow R, a\mapsto a^\infty$, such that $a\boxplus a^\infty=a^\infty$ holds for all $a\in R$. According to \cite{cortinas}, Proposition 2.3.1, infinite sum rings always have vanishing $K_0$-group. }: Choose a decomposition of $X$ into a countable number $X_0, X_1, X_2, \ldots$ of countably infinite subsets (i.\,e. their disjoint union is $X$). Now, choose four operators $\alpha_0, \beta_0, \alpha_1, \beta_1\in \B_{(1)}(\mathbb Q_p(X))$ such that the following properties hold: 
$\alpha_0$ is a bijection of $X_0$ onto $X$ (here, we identify the elements $x\in X$ with the corresponding elements $\delta_x\in\mathbb Q_p(X)$) and maps the elements of $X_1\cup X_2\cup\ldots$ to $0$; 
$\beta_0$ maps $X$ bijectively onto $X_0$; 
$\beta_1$ maps, for each $n\in\mathbb N$, the elements of $X_n$ bijectively onto $X_{n+1}$; 
$\alpha_1$ maps, for each $n\in\mathbb N, n\geq 1$, the elements of $X_n$ bijectively onto $X_{n-1}$ and those of $X_0$ to $0$. 
Furthermore, one requires that $\alpha_0\beta_0|_X=\operatorname{id}_X$ and $\alpha_1\beta_1|_X=\operatorname{id}_X$. 

 Operators fulfilling these requirements are easily verified to satisfy the relations 
\begin{align*}
\alpha_0\beta_0=\alpha_1\beta_1=1, \beta_0\alpha_0+\beta_1\alpha_1=1 
\end{align*}
that imply that $\B_{(1)}(\mathbb Q_p(X))$ is a sum ring.

But $\B_{(1)}(\mathbb Q_p(X))$ is even an infinite sum ring: that is, for each operator $a\in \B_{(1)}(\mathbb Q_p(X))$, define $a^\infty$ to be the operator in $\B(\mathbb Q_p(X))$ that acts as a diagonal operator on each $X_n$ as if it acted as $a$ on $X$, i.\,e. more precisely that maps $x\in X_n$ to $\beta_1^n\beta_0 a\alpha_0\alpha_1^n x$. The operator $a^\infty$ lies in $\B_{(1)}(\mathbb Q_p(X))$ because its matrix admits no entries in $\mathbb Q_p\setminus \mathbb Z_p$ (as does the matrix of $a\in \B_{(1)}(\mathbb Q_p(X))$). As one has $a^\infty=\sum_{n\in\mathbb N}\beta_1^n\beta_0 a\alpha_0\alpha_1^n$ (pointwise limit), it is easy to see that it always satisfies the equation 
\begin{align*}
\beta_0 a\alpha_0+\beta_1 a^\infty \alpha_1=a^\infty. 
\end{align*}
This fact implies indeed that $\B_{(1)}(\mathbb Q_p(X))$ is an infinite sum ring (because the map $a\mapsto a^\infty$ is a unital ring homomorphism) and that $K_0(\B_{(1)}(\mathbb Q_p(X)))=0$. 
\end{proof}

It remains to prove that $K_0(\B(\Q_p(X)))=0$ for a countable set $X$. In the sequel, we will always (without loss of generality) assume $X=\N$ for simplicity. It is sufficient to show that each idempotent $e\in \B(\Q_p(\N))$ is stably equivalent to zero because $M_m(\B(\Q_p(\N)))\cong \B(\Q_p(\N\times\{1, \ldots, m\}))\cong \B(\Q_p(\N))$ for all $m\in \N, m\geq 1$. Our strategy will be the following: First, we construct an idempotent $f\in \B(\Q_p(\N))$ with the finite-dimensional image $\operatorname{im}f=\Q_pe_1+\ldots+\Q_pe_n$ (where the columns $e_1, \ldots, e_n$ of $e$ are chosen in such a way that they contain all entries of $e$ in $\Q_p\setminus\Z_p$) and with the further property that also $g=e-f$ is an idempotent with $\operatorname{im} g\subseteq\operatorname{im} e$ and $\|g\|\leq 1$. Second, we show the stable equivalence of $g$ and $e$ (which proves the result because of Lemma \ref{infinite sum ring}).

In the first step, we want to show that finite-dimensional subspaces have a complement: 

\begin{lemma}\label{finite-dimensional subspaces and their complement}
Let $e\in\B(\Q_p(\N))$ be an idempotent and define $U=e\Q_p(\N)$. Furthermore, let $V\subseteq U\cap c_0(\N, \Q_p)$ be a finite-dimensional $\Q_p$-vector space. Then, there exists a $\tau$-continuous $\|\ \|$-contractive idempotent endomorphism $\tilde f\colon U\rightarrow U$ such that $\operatorname{im}\tilde f=V$. 
\end{lemma}

\begin{proof} Choose a basis $(\tilde v_1, \ldots, \tilde v_m)$ for $V$. Using certain operations (addition of a multiple of a basis vector to another, multiplication of a basis vector with a number), it is possible to transform this basis into a basis $(v_1, \ldots, v_m)$ of $V$ such that $\|v_1\|=\ldots=\|v_m\|=1$ and with the property that for each $k=1, \ldots, m$, there is a number $a_k\in\N$ such that $v_i(a_k)=\delta_{ik}$ for $i=1, \ldots, m$ (where $v_i(a_k)$ is the $a_k$-th entry of $v_i$). 

 For an element $\xi\in U$, define now $\tilde f(\xi)=\xi(a_1)v_1+\ldots+\xi(a_m)v_m$. The function $\tilde f\colon U\rightarrow U$ satisfies the required properties of the lemma. 
\end{proof}

 Now, we can proceed to the announced decomposition of the idempotent $e$: 

\begin{lemma}\label{decomposition of idempotents}
Let $e\in\B(\Q_p(\N))$ be an idempotent and define $U=e\Q_p(\N)$. Let $e_i, i\in\N,$ be the columns of $e$ (considered as a matrix) and choose $n\in\N$ such that $e_i$ does not contain entries in $\Q_p\setminus\Z_p$ for $i>n$. Then, there is an idempotent $f\in \B(\Q_p(\N))$ with the properties that $fe=ef=f$, that $\operatorname{im} f$ is the finite-dimensional $\Q_p$-vector space $\Q_pe_1+\ldots+\Q_pe_n$ and that $e-f$ is an idempotent with $\|e-f\|\leq1$. 
\end{lemma}

\begin{proof} Define $V$ to be the space $\Q_pe_1+\ldots+\Q_pe_n\subseteq U$ and apply the preceding lemma on it. Let $\tilde f\colon U\rightarrow U$ be the $\|\ \|$-contractive $\tau$-continuous idempotent of the preceding lemma with $\operatorname{im}f=V$. Define $f=\tilde f\circ e$. It is clear that $f$ is $\tau$-continuous (and therefore in $\B(\Q_p(\N))$) and that it is an idempotent with $\operatorname{im}f=V$. The equation $ef=fe=f$ follows from $V\subseteq U$. Furthermore, we obtain $(e-f)^2=e-fe-ef+f=e-f-f+f=e-f$ and $e-f$ is thus an idempotent.

We still have to show that $\|e-f\|\leq 1$. Let $k$ be in $\{0, \ldots, n\}$. A calculation yields 
\begin{align*}
(e-f)\delta_k=e\delta_k-(\tilde f\circ e)\delta_k=e_k-\tilde f e_k=e_k-e_k=0. 
\end{align*}
On the other hand, for $k\in\N, k>n$, we obtain 
\begin{align*}
\|(e-f)\delta_k\|=\|e\delta_k-f\delta_k\|=\|e_k-\tilde fe_k\|\leq\|e_k\|\vee\|\tilde fe_k\|\leq1 
\end{align*}
because $\tilde f$ is $\|\ \|$-contractive and $\|e_k\|\leq1$.

Hence, considered as a matrix, $e-f$ contains no entries in $\Q_p\setminus\Z_p$ and we have shown $\|e-f\|\leq1$. 
\end{proof}

\begin{lemma}\label{close projections and approximation}
Let $A$ be a closed sub-$\mathbb Z_p$-algebra of $\B(\mathbb Q_p(X))$ that contains the identity. Let $e \in A$ be an idempotent such that $e\neq 0$ and $a\in A$ such that $\|e-a\|<1/\|e\|^3$. Then, the sequence $(P_m(a))_{m\in\N}$ converges to an idempotent $e_a\in A$ that is equivalent to $e$. 
\end{lemma}

 The polynomials $P_m\in\Z[x], m\in\N, m\geq1$, have been defined in the proof of Lemma \ref{general lemma on approximation}. \\

\begin{proof} First, we obtain $\|e\|\geq1$ and $\|e-a\|<1/\|e\|^3\leq1\leq\|e\|$ and hence $\|a\|=\|e\|\geq1$. 
 Now, notice that $((e-a)+a)^2=(e-a)+a$, i.\,e. $(e-a)^2+a^2+(e-a)a+a(e-a)=(e-a)+a$ or 
\begin{align*}
\|a^2-a\|=\|(e-a)-(e-a)^2-(e-a)a-a(e-a)\|\leq\|e-a\|\|a\|<1/\|a\|^2. 
\end{align*}
Recall from the proof of Lemma \ref{general lemma on approximation} that the sequence $P_m(a)$ converges to an idempotent element $e_a\in A$ such that $\|a-e_a\|<1/\|a\|=1/\|e\|$. We therefore obtain that also $\|e-e_a\|<1/\|e\|$ holds. According to Lemma \ref{close projections}, the idempotents $e$ and $e_a$ are equivalent with respect to $A$. 
\end{proof}

\begin{lemma}\label{idempotents with finite-dimensional image}
Let $f\in \B(\Q_p(\N))$ be an idempotent whose image is a finite-dimen-sional $\Q_p$-vector space. Then, $f$ is stably equivalent to zero. 
\end{lemma}

\begin{proof} The case $f=0$ is obvious; assume therefore $f\neq0$. 
 Observe that, as $f$ has a finite dimensional image, it must be a compact operator. Therefore, its entries converge to zero.
 
 Choose an element $a\in \B(\Q_p(\N))$ that has (considered as a matrix) only finitely many non-vanishing entries and satisfies $\|a-f\|<1/\|f\|^3$. Choose $n\in\N$ such that the entry $a_{ij}$ of $a$ is zero if $i>n$ or $j>n$, i.\,e. such that $a\in M_{\{0, \ldots, n\}}(\Q_p)\subseteq \B(\Q_p(\N))$. On the one hand, the polynomials $P_m(a)$ will, according to Lemma \ref{close projections and approximation}, converge in norm to an idempotent $e_a\in \B(\Q_p(\N))$ that is equivalent to $f$. On the other hand, as the polynomials $P_m$ have no constant term, $P_m(a)$ never leaves the set $M_{\{0, \ldots, n\}}(\Q_p)\subseteq \B(\Q_p(\N))$ and hence, also $e_a$ has only finitely many non-vanishing entries. It is a well-known fact from linear algebra that $e_a$ is equivalent to a matrix of the form 
\begin{align*}
D=\begin{bmatrix} E_N & 0 \\ 0 & 0 \end{bmatrix} 
\end{align*}
where $E_N$ is the $N\times N$-unity matrix ($N\in\N$) and $0$ means vanishing matrices of appropriate size. At the end, we obtain that $D$ and $e$ are equivalent matrices and therefore, $e$ is stably equivalent to zero. 
\end{proof}

\begin{theorem}\label{K-theory for the continuous operators}
We have $K_0(\B(\Q_p(\N)))=0$. 
\end{theorem}

\begin{proof} Let $e$ be an idempotent in $M_m(\B(\Q_p(\N)))$. We want to show that $e$ is stably equivalent to zero. Because of the isomorphism $M_m(\B(\Q_p(\N)))\cong \B(\Q_p(\N\times\{1, \ldots, m\}))\cong \B(\Q_p(\N))$ for all $m\in \N, m\geq 1$, it is sufficient to treat the case $e\in \B(\Q_p(\N))$. Consider the decomposition $e=f+(e-f)$ stemming from Lemma \ref{decomposition of idempotents}. As we have $f(e-f)=(e-f)f=0$, we obtain that the stable equivalence class of $e$ is exactly the sum of the stable equivalence classes of $f$ and of $e-f$. But the stable equivalence class of $f$ is zero according to Lemma \ref{idempotents with finite-dimensional image} and the stable equivalence class of $e-f$ is zero as well because of Lemma \ref{infinite sum ring} (since $\|e-f\|\leq1$). Hence, the proof is finished. 
\end{proof}

\subsection{Lifting of idempotents in $\B(\Q_p(X))/\K(\Q_p(X))$}\label{Calkin lifting}

\begin{theorem}\label{lifting of idempotents}
Let $A\subseteq \B_{(1)}(\Q_p(\N))$ be a norm-closed subalgebra containing the set $\K_{(1)}(\Q_p(\N))$ of contractive compact operators. If $E$ is an idempotent element in the quotient algebra $A/\K_{(1)}(\Q_p(\N))$, then it has an idempotent lift $e$ in $A$, i.\,e. $e^2=e\in A$ and $e+\K_{(1)}(\Q_p(\N))=E$. 
\end{theorem}

\begin{proof} Choose an arbitrary lift $a\in A$ of $E$. Then, we get $a^2-a\in \K_{(1)}(\Q_p(\N))$ and also $a^n-a=(a^{n-2}+\ldots+a+1)(a^2-a)\in \K_{(1)}(\Q_p(\N))$ for $n>2$. Observe that there is a number $N\in \N$ such that for all $n\in \N$, the entries of $a^n-a$ (considered as a matrix) at the positions $(i, j)\in\N^2\setminus\{0, \ldots, N\}^2$ have absolute value smaller than $1$ (because the entries of $a^2-a$ converge to zero and one can write $a^n-a=(a^2-a)b_n=b_n(a^2-a)$ for an operator $b_n$ with $\|b_n\|\leq1$).

 Therefore, there must be $m, n\in\N$ with $n<m$ such that $\|a^m-a^n\|<1$: For $i, j\in\N$ such that $i>N$ or $j>N$, the entries of $a^n-a$ and $a^m-a$ (hence of $a^m-a^n$) at the position $(i, j)$ have absolute value smaller than $1$ anyway and for the finitely many positions in $\{1, \ldots, N\}^2$, the entries of $a^m-a$ and $a^n-a$ become arbitrarily close for certain $m, n$ for compactness reasons (we used $a^m-a^n=(a^m-a)-(a^n-a)$). Now, choose $k\in\N$ such that $k(m-n)>n$. Then, we also have $\|a^{(k+1)(m-n)}-a^{k(m-n)}\|<1$ and thus $\|a^{2k(m-n)}-a^{k(m-n)}\|<1$.
 
 Finally, apply our usual technique: As $b=a^{k(m-n)}$ and its square have distance less than $1$, the sequence of polynomials $P_l(b)$ (defined in the proof of Lemma \ref{general lemma on approximation}) converges to an idempotent $e$ for $l\rightarrow\infty$ that has distance less than $1$ from $b$. As all the operators $P_l(b)-b$ are of the form $(b^2-b)Q(b)$ (where $Q$ is a polynomial with coefficients in $\Z$), they are all compact, as well as $b-a$ and therefore $P_l(b)-a$. As the ideal of compact operators $\K_{(1)}(\Q_p(\N))$ is norm-closed in $\B_{(1)}(\Q_p(\N))$, we obtain that also $e-a$ is compact, i.\,e. $e$ is an idempotent lift of $E$. Note that in the whole procedure, we did not leave the algebra $A$ (even if it is non-unital) because we assumed it to be norm-closed, i.\,e. $e\in A$. That finishes the proof. 
\end{proof}

 As the operators in $\B(\Q_p(\N))$ have (considered as matrices) only finitely many entries not in $\Z_p$ and differ therefore only by a compact difference from operators in $\B_{(1)}(\Q_p(\N))$, we easily get the following corollary: 

\begin{corollary}\label{lifting of idempotents, great}
Let $A\subseteq \B(\Q_p(\N))$ be a norm-closed subalgebra containing the set $\K(\Q_p(\N))$ of compact operators. If $E$ is an idempotent element in the quotient algebra $A/\K(\Q_p(\N))$, then it has an idempotent lift $e$ in $A$, i.\,e. $e^2=e\in A$ and $e+\K(\Q_p(\N))=E$. 
\end{corollary}

\section*{Acknowledgments}

This research was supported in part by the ERC Consolidator Grant No. 681207. We thank the referee for useful comments that improved the exposition of the results. The results presented in this paper are part of the PhD project of the first author.


\nocite*{}



\end{document}